\documentclass[reqno,12pt]{amsart}

\usepackage[margin=2.5cm]{geometry}
\usepackage[T1]{fontenc}
\usepackage[utf8]{inputenc}
\usepackage[english]{babel}
\usepackage{amsmath,amssymb,amsthm,amsrefs}
\usepackage{ulem}

\usepackage{enumitem}
\usepackage{mymacros}
\usepackage{hyperref}
\usepackage{xcolor}
\usepackage{orcidlink}
\usepackage{comment}
\usepackage{tikz}

\usepackage{centernot}
\usepackage{mathtools}

\usepackage{fouriernc}

\title{Interpolation and random interpolation in de Branges-Rovnyak spaces}

\author[A. Hartmann]{Andreas Hartmann}
\address{Univ. Bordeaux, CNRS, Bordeaux INP, IMB, UMR 5251, F-33400 Talence, France}
\email{andreas.hartmann@math.u-bordeaux.fr}

\author[G. Lamberti]{Giuseppe Lamberti}
\address{Univ. Bordeaux, CNRS, Bordeaux INP, IMB, UMR 5251, F-33400 Talence, France} 
\email{giuseppe.lamberti@math.u-bordeaux.fr}

\thanks{This work was partially supported by ANR-24-CE40-5470.}


\begin{document}

\begin{abstract}
The aim of this paper is to characterize universal and multiplier interpolating sequences for de Branges-Rovnyak spaces $\cH(b)$ where the defining function $b$ is a general non-extreme rational function.
Our results carry over to recently introduced higher order local Dirichlet spaces and thus generalize previously known results in classical local Dirichlet spaces. In this setting, we also investigate random interpolating sequences with prescribed radii, providing a $0-1$ law.
\end{abstract}

\maketitle

\section{Introduction}
Let $\cH$ be a reproducing kernel Hilbert space on a domain $\Omega \subset \bC^d$ and let $k$ be its reproducing kernel. Using the notation $k_{\lambda}(z)=k(\lambda,z)$, $\lambda,z\in\Omega$, we say that a sequence $\Lambda=(\lambda_n)_{n \in \bN}$ on $\Omega$ is \textit{(universal) interpolating} for $\cH$ (or $\cH$-interpolating) if the operator
\begin{align*}
    R_\Lambda \colon \cH & \longrightarrow  \ell^2  \\
    f & \mapsto  \Big( \frac{f(\lambda_n)} {\|k_{\lambda_n}\|_{\cH}} \Big)_{n \in \bN} 
\end{align*}
is bounded and surjective. We need to introduce another notion of interpolation. For this purpose we recall the \textit{multiplier algebra} of $\cH$ as
\[
\cM (\cH) := \{ \varphi:\Omega \to \mathbb C : \varphi f \in \cH \text{ for all } f \in \cH \}.
\]
We mention, among other properties of this algebra, that every function that belongs to $\cM (\cH)$ must be bounded on $\Omega$. We say that a sequence $\Lambda \subset \Omega$ is \textit{multiplier interpolating} for $\cH$ if for every sequence $(a_n)_{n \in \bN} \in \ell^\infty$ there exists $\varphi \in \cM (\cH)$ such that $\varphi(\lambda_n)=a_n$ for every $n \in \bN$. In the literature these sequences are also called interpolating for $\cM (\cH)$ or $\cM(\cH)$-interpolating.

In general, every multiplier interpolating sequence is also a universal interpolating sequence (see \cite{seip04}*{Chapter 2, Proposition 1}), although the converse is not always true (e.g. the Bergman space --- this is common knowledge today, but it was first proved by D. Amar and \'E. Amar \cite{Amar78} in the late 70's). However the reverse implication does hold for any reproducing kernel Hilbert space with the Nevanlinna-Pick property. For a recent account on this theory, we refer the reader to \cite{Agler02}.

We now proceed to describing more in details the setting of our work. Consider the open unit disk $\bD:= \{z \in \bC: |z|<1\}$ and its topological boundary $\bT:=\partial \bD$. The \textit{pseudohyperbolic distance} $\rho$ is defined as
\[
\rho(z,w) = \left| \frac{z-w}{1-\ol{w}z} \right| \quad z,w \in \bD.
\]
An example of a space where the two notions of interpolating sequences introduced before are equivalent is the Hardy space $H^2$, the space of holomorphic function $f$ on $\bD$ such that
\[
\|f\|^2_{H^2}:= \sup_{0<r<1} \int_\bT |f(r\zeta)|^2 dm(\zeta) < \infty,
\]
where $dm$ is the normalized Lebesgue measure on $\bT$. It is not difficult to see that the multiplier algebra of $H^2$ is precisely $H^\infty$, the space of bounded holomorphic functions on $\bD$.\\
In \cite{carleson58}, Carleson proved that a sequence $\Lambda$ is interpolating for $H^\infty$ if and only if it satisfies the so-called \textit{Carleson condition}, i.e.
\begin{equation} \tag{C} \label{Carleson_condition}
    \inf_{n \in \bN} \prod_{k \neq n} \rho(\lambda_n,\lambda_k) > 0.
\end{equation}
Later on Shapiro and Shields \cite{Shapiro61} proved that the same condition characterizes also interpolating sequences for $H^2$ (and more generally in $H^p$).

We now introduce the space of holomorphic functions that is the focus of our study. Consider $b \in H^\infty$ such that $\|b\|_{H^\infty}\leq1$. The \textit{de Branges-Rovnyak space} $\cH(b)$ is defined as the reproducing kernel Hilbert space of holomorphic functions on $\bD$ with the reproducing kernel given by
\[
k_w^b (z) = \frac{1-\ol{b(w)}b(z)}{1-\ol{w}z} \quad z,w \in \bD.
\]

Another way of defining $\cH(b)$ is to consider it as the range space $(I-T_bT_{\bar{b}})^{1/2}H^2$, where $T_b:H^2\to H^2$ is the Toeplitz operator with symbol $b$: $T_bf=P_+(bf)$, and $P_+$ is the Riesz projection. 
Then a natural norm in $\cH(b)$ is  given by the range norm, but we will not explore this further here and introduce below an equivalent and more transparent norm in the setting we are interested in. These spaces were first introduced by de Branges and Rovnyak \cite{dBR66}. Since then, they have been well-studied and are regarded as fundamental within the theory of holomorphic functions, playing a significant role in various areas of function theory and operator theory. For a recent and thorough treatment of $\cH (b)$ spaces we refer the reader to \cite{Fricain15a}.

We recall that whenever $\|b\|_{H^\infty} < 1$, the space $\cH(b)$ actually coincides with $H^2$ equipped with an equivalent norm. For this reason we focus our attention to $\cH (b)$ spaces where the function $b$ has norm equal to $1$. In addition we require that $b$ is a rational function, and a \textit{non-extreme} point of the unit ball of $H^\infty$, the latter meaning that
\[
\int_\bT \log (1-|b(\zeta)|) dm(\zeta) > -\infty.
\]
When $b$ is non-extreme, then there exists a unique outer function $a \in H^\infty$, $\|a\|_{H^\infty}=1$, such that $a(0)>0$ and
\[
|a(\zeta)|^2 + |b(\zeta)|^2 = 1 \quad \text{for a.e. } \zeta \in \bT.
\]
We refer to $a$ as the \textit{Pythagorean mate} of $b$ and we call $(a,b)$ a \textit{pair}. By the Fej\'er-Riesz theorem, whenever $b$ is rational, then $a$ is also rational, and we will then call $(a,b)$ a \textit{rational pair}. From now on we always assume that $a$ and $b$ are outer. We mention that the Fej\'er-Riesz theorem also gives a recipe to construct the Pythagorean mate $a$ but this requires the localization of roots of polynomials of degree twice the degree of $b$. Here is an explicit example which can be found in \cite{Fricain16}: if $b(z)=(1-z)^2/4$ then $a(z)=c(1+z)(z-(3+2\sqrt{2}))$ for some suitable constant $c$. It it not difficult to see that if $(a,b)$ is a rational pair then $(a,b)$ is also a \textit{corona pair}, which means that
\[
\inf_{z \in \bD} \big(|a(z)|^2 + |b(z)|^2\big) > 0.
\]

Our main result strongly relies on the following description of this class of spaces. This was proved in \cite{Costara13}*{Lemma 4.3} (see also \cite{Blandigneres15} for another proof).

\begin{thm} \label{decomposition}
    Let $(a,b)$ be a rational pair and $\zeta_1, \ldots, \zeta_l$ the zeroes of $a$ on $\bT$ with multiplicity $m_1, \ldots, m_l$. Then
    \begin{equation} \label{eq:decomposition}
    \cH(b) = \Big( \prod_{j=1}^l (z-\zeta_j)^{m_j} \Big) H^2 \oplus P_{N-1},
    \end{equation}
    where $N=\sum_j m_j$ and $P_{N-1}$ is the space of polynomials of degree at most $N-1$.
\end{thm}
With the symbol $\oplus$ we mean a direct sum, but not necessarily orthogonal (see \cites{Fricain16,Fricain18}). Note that since $a$ is an outer function, the first summand in \eqref{eq:decomposition} corresponds to $aH^2$.
Let $M(\bar a):= T_{\bar{a}}H^2$ be the range space of the Toeplitz operator $T_{\ol{a}}$ equipped with the range norm, then, using \cite{Fricain15a}*{Theorems 23.2 and 28.7}, it can be shown that $\cH(b) = M(\ol a)$ as sets if and only if $(a, b)$ forms a corona pair. By standard results in functional analysis, the corresponding norms are equivalent. Moreover, when $\cH(b)$ is given by $\eqref{eq:decomposition}$, we can introduce an equivalent norm on $\cH(b)$: if $f=\prod(z-\zeta_j)^{m_j}g+p$, then
\[
\|f\|^2_{\cH(b)}:= \|g\|^2_{H^2} + \|p\|^2_{H^2},
\]
which is equivalent to the earlier mentioned range norm. We will use this also as a norm for $M(\bar a)$.

Interpolating sequences, both multiplier and universal, for de Branges-Rovnyak spaces have previously been studied for certain specific choices of $b$. The objective of this paper is to extend those results. To be more precise, in \cite{Serra03} the author provides a complete characterization of multiplier interpolating sequences for 
$\cH(b)$ when this space can be viewed as a \textit{local Dirichlet space} $\cD_\mu$ (for a precise definition, we refer the reader to Section 2). Later, it was shown in \cite{Chacon11} that, under these conditions, multiplier and universal interpolating sequences coincide with sequences that satisfy a separation and a Carleson measure condition. It is worth noting that the space $\cD_\mu$ has the complete Nevanlinna-Pick property, while, more generally, determining when a de Branges-Rovnyak space has this property is a more delicate question (see \cite{Chu20}).\\

The first main result of this paper is the following theorem.

\begin{thm} \label{main_thm}
    Let $b$ be a non-extreme and rational function of the unit ball of $H^\infty$ and $a$ its Pythagorean mate. Let $\zeta_1, \ldots, \zeta_l$ be the zeroes of $a$ on $\bT$ with multiplicity $m_1, \ldots, m_l$ and $\Lambda$ be a sequence on $\bD$. The following are equivalent.
    \begin{enumerate}[label=(\roman*)]
        \item $\Lambda$ is $\cM (\cH (b))$-interpolating.
        \item $\Lambda$ is $\cH (b)$-interpolating.
        \item $\Lambda$ satisfies the Carleson condition and 
        \begin{equation} \label{main_sum}
        \sum_{n \in \bN} \frac{1-|\lambda_n|^2}{|\zeta_j-\lambda_n|^{2 m_j}} < \infty \quad \text{for every } 1\leq j \leq l.
        \end{equation}
    \end{enumerate}
\end{thm}
We mention that \eqref{main_sum} already appears in previous work on existence of non-tangential boundary values in model spaces $K_B=\cH(B)$ when $B$ is a Blaschke product (see \cite{Ahern70}; see also \cite{Fricain08} for an extension to $\cH(b)$ on the upper half plane). 

We now introduce the random setting we are interested in. If not differently specified, by the term \textit{random sequence} we mean a sequence of the following kind. Let $(r_n)_{n \in \bN}$ be a sequence of radii such that $0 \leq r_n < 1$ for every $n$ and consider a family $(\vartheta_n)_{n \in \bN}$ of independent random variables uniformly distributed in $[0,2\pi]$. Define
\[
\Lambda (\omega) = (\lambda_n(\omega))_{n \in \bN} = (r_ne^{i \vartheta_n (\omega)})_{n \in \bN}.
\]
Such random sequences, sometimes mentioned in the literature as Steinhaus sequences, have already been studied in the last years. We mention the pioneering works of Cochran \cite{cochran90} and Rudowicz \cite{rudowicz94} leading to a $0-1$ law for interpolating sequences for the Hardy space and the work of Bogdan \cites{bogdan96} where this model was used to better understand random zero sets in Dirichlet spaces. It is worth mentioning more recent works, such as \cite{chalmoukis22} which deals with interpolation in Dirichlet spaces, \cites{dayan2023,chalmoukis2024}, on interpolation and related properties for spaces of holomorphic functions in several variables, and \cite{Lamberti24} which deals with random interpolation in the Nevanlinna class.

We prove the following theorem.

\begin{thm} \label{main_random}
    Let $\Lambda$ be a random sequence and $b$ a non-extreme, rational function. Consider the decomposition of $\cH(b)$ as described in \eqref{eq:decomposition} and define
    \[
    M=\max (\{ m_j : 1\leq j \leq l \}).
    \]
    Then
    \[
    \bP ( \Lambda \text{ is interpolating for } \cH(b) ) = 
    \begin{dcases}
        1 &
\text{if } \sum_{n \in \bN} (1-|\lambda_n|^2)^{\frac{1}{2M}} < \infty \\
   0 &\text{if }     \text{otherwise}
 \end{dcases}
    \]
    where with the term interpolating we mean both multiplier and universal interpolating.
\end{thm}

Since the event of being interpolating is a so-called tail event, in view of Kolmogorov's 0-1 law the above statements translate into if and only if statements.
The reader should also notice that the above condition only involves the highest multiplicity $m_j$ but not the number of zeros $\zeta_k$.

\begin{rem}
In the existing literature, the analysis of such random sequences typically involves introducing a dyadic partition of the unit disk along with an associated counting function \cites{chalmoukis22, cochran90, chalmoukis2024}. Specifically, for each integer $k \geq 0$, define
\begin{align*}
A_k & = \left\{ z \in \bD : 1-2^k \leq |z| < 1-2^{-(k+1)} \right\}, \\ 
N_k & = \# (A_k \cap \Lambda).
\end{align*}
Note that $N_k$ is actually a number and not a random variable, since it depends only on the sequence of the {\it a priori} fixed radii. Then the condition $\sum_{n} (1-|\lambda_n|)^{\frac{1}{2M}} < \infty$ is equivalent to $\sum_k N_k 2^{-\frac{k}{2M}}<\infty$.
\end{rem}

\subsection{Notation} If $f$ and $g$ are positive expressions, we will write $f \lesssim g$ if there exists $C>0$ such that $f\leq C g$, where $C$ does not depend on the parameters behind $f$ and $g$. We will simply write $f \simeq g$ if $f \lesssim g$ and $g \lesssim f$. Finally when $f$ and $g$ are expressions for which we can consider the limit of their quotient in a point, with $f \sim g$ we mean that $\lim f/g =1$ in that point.

\section*{Acknowledgments}
We gratefully acknowledge the referees for their careful reading and their very insightful comments which greatly helped to correct and improve the initial version of the text.

\section{Preliminaries}

\subsection{de Branges-Rovnyak spaces and local Dirichlet spaces}
We start this section describing the existing connection between de Branges-Rovnyak and local Dirichlet spaces, that was already mentioned in the introduction. Let $\mu$ be a finite positive Borel measure on $\bT$ and let $P[\mu]$ its Poisson extension. The associated \textit{generalized Dirichlet space} $\cD_\mu$ consists of those $f \in H^2$ such that
\[
D_\mu (f) := \int_{\bD} |f'(z)|^2 P[\mu](z) dA(z) < \infty.
\]
In particular whenever $\mu=\delta_\zeta$ for some $\zeta \in \bT$, the space is usually called \textit{local Dirichlet space}. In this situation we have that (see \cite{Richter91b})
\[
D_{\delta_\zeta} (f) = \int_{\bT} \Big| \frac{f(\omega)-f(\zeta)}{\omega-\zeta} \Big|^2 dm(\omega).
\]
We will write $\cD_\zeta$ instead of $\cD_{\delta_\zeta}$ and $D_\zeta$ instead of $D_{\delta_\zeta}$ in the rest of the paper to simplify the notation.

The spaces $\cD_\mu$ were first introduced by Richter \cite{Richter91a}. Later on it was shown in \cite{Richter91b} that $\cD_{\zeta}$ can be viewed as a de Branges-Rovnyak space for some $b$ with equivalence of norms. This result was strengthened by Sarason \cite{Sarason97}, who showed how to choose $b$ in order to have $\cD_{\zeta} = \cH(b)$ with equality of norms. The converse of this result was then proved in \cite{Chevrot10}. Finally in \cite{Costara13}, the authors considered the problem of when $\cD_\mu = \cH(b)$ without necessarily equality of norms. We report here one of the main result of their paper.

\begin{thm}[\cite{Costara13}*{Theorem 4.1}]
    Let $(a,b)$ be a rational pair and let $\mu$ be a finite positive measure on $\bT$. Then $\cD_\mu=\cH(b)$ if and only if:
    \begin{enumerate}[label=(\roman*)]
        \item the zeroes of $a$ on $\bT$ are all simple;
        \item the support of $\mu$ is exactly equal to this set of zeros.
    \end{enumerate}
\end{thm}

In \cite{Richter21} the authors introduce the \textit{local Dirichlet spaces of higher order} $\cD_\zeta^N$, with $N \in \bN$. For $f \in H^2$ define the (possibly infinite) expression
\[
D_\zeta^N(f):= \inf \Big\{ \big\| \frac{f-p}{(z-\zeta)^N} \big\|_{H^2} : p \text{ is a polynomial of degree strictly less than } N \Big\},
\]
and $\cD_\zeta^N$ is the space of functions $f \in H^2$ such that $\|f\|_{\cD_\zeta^N}:=\|f\|_{H^2} + D_\zeta^N(f) < \infty$. In particular, if $f$ extends analytically to a neighborhood of $\zeta$, we have (see \cite{Richter21})
\[
D_\zeta^N (f) := \int_\bT \Big| \frac{f(\omega)-T_{N-1}(f,\zeta)(\omega)}{(\omega-\zeta)^N} \Big|^2 dm(\omega) < \infty,
\]
where $T_{N-1}(f,\zeta)$ is the $(N-1)$-th order Taylor polynomial of $f$ at $\zeta$,
\[
T_{N-1}(f,\zeta)(\omega) = \sum_{j=0}^{N-1} \frac{f^{(j)}(\zeta)}{j!} (\omega-\zeta)^j.
\]
In the same paper the authors also proved that these spaces can be seen as de Branges-Rovnyak spaces. The theorem is the following.

\begin{thm}[\cite{Richter21}*{Theorem 1.1}]
    Let $b$ be a non-extreme point of the unit ball of $H^\infty$ with $b(0)=0$, and let $N \in \bN$. Then the following are equivalent:
    \begin{enumerate}[label=(\roman*)]
        \item $b$ is a rational function of degree $N$ such that the mate $a$ has a single zero of multiplicity $N$ at a point $\zeta \in \bT$;
        \item there is a $\zeta \in \bT$ and a polynomial $\tilde{p}$ of degree strictly less than $N$ with $\tilde{p}(\zeta) \neq 0$ such that $\|f\|^2_{\cH(b)} = \|f\|_{H^2}^2 + D_\zeta^N(\tilde{p}f)$ for $f\in \cH(b)$.
    \end{enumerate}
    If these conditions hold, then there are polynomials $p$ and $q$ of degree less or equal to $N$ such that $b=\frac{p}{q}$, $a=\frac{(z-\zeta)^N}{q}$, $\tilde{p}(z)=z^N \ol{p(1/\ol{z})}$, for $z \in \bD$, and $|q(\omega)|^2=|p(\omega)|^2+|\omega-\zeta|^{2N}$ for all $\omega \in \bT$. Furthermore, $\cH(b) = \cD_\zeta^N$ with equivalence of norms.
\end{thm}
We refer the reader to \cite{Pouliasis25} for some recent related results on the equality $\cD_{\mu}=\cH(b)$.

To complete the preliminaries needed on de Branges-Rovnyak spaces, we recall the following description of the multiplier algebra of $\cH(b)$ spaces, proved in \cite{Fricain19}*{Proposition 3.1}.

\begin{thm} \label{multiplier_algebra}
    Let $b\in H^{\infty}$ be a rational, non inner function with $\|b\|_{H^{\infty}}=1$. Then $\cM(\cH(b))= H^\infty \cap \cH(b)$.
 \end{thm}

 It is worth pointing out that the aforementioned theorem was established for the space $M(\ol a)$
($a$ being the Pythagorean mate of $b$); however, as already mentioned, in our context this space coincides with $\cH(b)$. In general, $ H^\infty \cap \cH(b)$ does not need to be an algebra (see \cite{Fricain19} for references).\\

We will need another result concerning the derivatives of Blaschke products. This was proved in \cites{Ahern71,Ahern70}, but we report here the version presented in \cite{Fricain15a}*{Theorem 21.8}.

\begin{thm} \label{derivative B_Lambda}
    Let $(\lambda_n)_{n \in \bN}$ be a Blaschke sequence in $\bD$, and let $B_\Lambda$ be the corresponding Blaschke product. Assume that, for an integer $N\geq 0$ and a point $\zeta \in \bT$, we have
    \begin{equation}\label{ahern-clark}
    \sum_{n=0}^\infty \frac{1-|\lambda_n|}{|\zeta-\lambda_n|^{N+1}} \leq A.
    \end{equation}
    Then the following hold.
    \begin{enumerate}[label=(\roman*)]
        \item For each $0 \leq j \leq N$, both limits
        \[
        B_\Lambda^{(j)}(\zeta):= \lim_{r \to 1^-} B_\Lambda^{(j)}(r\zeta) \quad \text{and} \quad \lim_{r \to 1^+} B_\Lambda^{(j)}(r\zeta) 
        \]
        exist and are equal.
        \item There is a constant $C=C(N,A)$ such that the estimation
        \[
        |B_\Lambda^{(j)}(r\zeta)| \leq C
        \]
        uniformly holds for $r \in [0,1]$ and $0 \leq j \leq N$.
    \end{enumerate}
\end{thm}

From the proof presented in \cite[Theorem 21.8]{Fricain15a}, it can also be deduced that, under the condition \eqref{ahern-clark}, the higher order derivatives up to the order $N$ of the partial Blaschke products $B_n=\prod_{k=1}^n\frac{|\lambda_k|}{\lambda_k}\frac{\lambda_k-z}{1-\overline{\lambda_k}z}$ converge at $\zeta$ to the corresponding derivatives of $B_\Lambda$, when $n \to \infty$. 

\subsection{Probability background}

On the probabilistic side, we need the following classical theorem (see for instance \cite{billingsley95}).

\begin{thm}[Kolmogorov's three series theorem] \label{kolmogorov 3-series}

    Let $(X_n)_{n \in \bN}$ be a sequence of independent random variables and for any $c>0$ let be $Y_n := X_n \chi_{\{|X_n| \leq c\}}$. Consider the series
    \begin{enumerate}[label=(\roman*)]
        \item $\sum_n \bP(|X_n| > c)$
        \item $\sum_n \bE[Y_n]$
        \item $\sum_n \bV[Y_n]$.
    \end{enumerate}
    In order that $\sum_n X_n$ converges almost surely it is necessary that the three series converge for all positive $c$ and sufficient that they converge for some positive $c$.
\end{thm}

\section{Interpolating sequences: proof of Theorem \ref{main_thm}}

Referring to Theorem \ref{main_thm}, since we already know that $(i) \implies (ii)$, the plan of the proof is to establish $(ii) \implies (iii) \implies (i)$. We start proving $(ii) \implies (iii)$. In this direction, we mention the following lemma, that holds for any choice of the function $b$ in the unit ball of $H^{\infty}$. It was first proved in \cite{Fricain05}, see also \cite{Fricain15a}*{Theorem 31.16}.

\begin{lem} \label{H(b) implies Carleson}
    If $\Lambda$ is interpolating for $\cH(b)$ then $\Lambda$ satisfies the Carleson condition.
\end{lem}

The following lemma, on the other hand, clearly relies on the description of the $\cH(b)$ spaces provided in \eqref{eq:decomposition}.

\begin{lem} \label{interp_implies_sum}
   Let $b$ be a rational, non inner function with $\|b\|_{H^{\infty}}=1$. If $\Lambda$ is interpolating for $\cH (b)$, then $\Lambda$ satisfies \eqref{main_sum}.
\end{lem}

\begin{proof} 
Observe first that since $\Lambda$ is interpolating for $\cH(b)$, the restriction operator $R_{\Lambda}$ is bounded. Moreover, since $b$ is non extreme, $\cH(b)$ contains the polynomials  and in particular the constant function $1$. Hence 
\[
 R_{\Lambda}1=\Big( \frac{1} {\|k_{\lambda_n}^b\|_{\cH(b)}} \Big)_{n \in \bN} \in \ell^2.
\]
In particular, any idempotent sequence $(e_n)_{n\in\bN}$, $e_n\in\{0,1\}$, can be interpolated by a function in $\cH(b)$. Clearly, as an interpolating sequence of $\cH(b) \subset H^2$, $\Lambda$ also satisfies the Blaschke condition.

Suppose now that \eqref{main_sum} does not hold for some $r$, $1\leq r \leq l$. 
Define 
    \[
    u_n := \frac{1-|\lambda_n|^2}{|\zeta_{r}-\lambda_n|^{2m_r}}.
    \]
    Then there exists a subsequence $(\lambda_{n_k})_{k \in \bN}$ of $\Lambda$ such that $\lambda_{n_k} \to \zeta_{r}$ as $k \to \infty$. In view of the Blaschke condition, we can still assume that $(u_{n_k})_{k \in \bN}$ does not belong to $\ell^1$. We can divide this subsequence into $(\lambda_{n_k,1})_{k \in \bN}$ and $(\lambda_{n_k,2})_{k \in \bN}$ such that both $u_{n_k, 1}$ and $u_{n_k,2}$ are not in $\ell^1$. Since $\Lambda$ is $\cH (b)$-interpolating, there exists $f \in \cH(b)$ such that
    \begin{align*}
        f(\lambda_{n_k,1}) & = 1 \\
        f(\lambda_{n_k,2}) & = 0.
    \end{align*}
    Since both $u_{n_k, 1}$ and $u_{n_k,2}$ are not in $\ell^1$, we have
    \[
    \sum_{k>k_0} u_{n_k,1} = \infty \quad \text{and} \quad \sum_{k>k_0} u_{n_k,2} = \infty.
    \]
    From Theorem \ref{decomposition} we know that 
    \[
    f(z) = \prod_{j=1}^l (z-\zeta_j)^{m_j} g(z) + \sum_{j=0}^{N-1}c_j(z-\zeta_{r})^j = \prod_{j=1}^l (z-\zeta_j)^{m_j} g(z) + p_{r}(z)
    \] 
    with $g \in H^2$.\\
    Suppose $c_0=0$. Since $\lambda_{n_k,1} \to \zeta_{r}$, there exists $k_1 \in \bN$ such that $|p_{r}(\lambda_{n_k,1})| \leq 1/2$, for every $k >k_1$, which yields
    \begin{equation} \label{c_0=0}
    \Big| 1-\sum_{j=0}^{N-1}c_j(\lambda_{n_k,1}-\zeta_{r})^j \Big| \geq \frac{1}{2} \quad \forall k>k_1.
    \end{equation}
    From Lemma \ref{H(b) implies Carleson} we know that $\Lambda$ satisfies \eqref{Carleson_condition}, which means that $\Lambda$ is $H^2$-interpolating and in particular $R_\Lambda$ is bounded on $\cH = H^2$. Recall that $N=m_1+\ldots+m_l$. Since $g \in H^2$, we get
    \begin{align*}
        \infty & > \sum_{k \in \bN} (1-|\lambda_{n_k,1}|^2)|g(\lambda_{n_k,1})|^2 \\
        & = \sum_{k \in \bN} (1-|\lambda_{n_k,1}|^2) \frac{|1-p_{r}(\lambda_{n_k,1})|^2}{\prod_{j=1}^l|\lambda_{n_k,1}-\zeta_j|^{2m_j}} \\
        & \geq \frac{C^2}{2^{2(N-m_r)}} \sum_{k > k_0} \frac{1-|\lambda_{n_k,1}|^2}{|\lambda_{n_k,1}-\zeta_{r}|^{2m_{r}}},
    \end{align*}
    where the last inequality follows from $|\lambda_{n_k,1}-\zeta_j| \leq 2$ and \eqref{c_0=0}. This gives a contradiction.\\
    If $c_0 \neq 0$ then, again because $\lambda_{n_k,2} \to \zeta_r$, there exists $k_2 \in \bN$ such that
    \[
    \Big| \sum_{j=1}^{N-1}c_j(\lambda_{n_k,2}-\zeta_{r})^j \Big|\leq  \frac{|c_0|}{2} \quad \forall k> k_2,
    \]
    which yields $|p_{r}(\lambda_{n_k,2})| \geq |c_0|/2$ for every $k >k_2$.
    Again, since $\Lambda$ is $H^2$-interpolating, we have 
    \begin{align*}
        \infty & > \sum_{k \in \bN} (1-|\lambda_{n_k,2}|^2)|g(\lambda_{n_k,2})|^2 \\
        & = \sum_{k \in \bN} (1-|\lambda_{n_k,2}|^2) \frac{|p_{r}(\lambda_{n_k,2})|^2}{\prod_{j=1}^l |\lambda_{n_k,2}-\zeta_j|^{2m_j}} \\
        & \geq \frac{C^2}{2^{2(N-m_r)}}\sum_{k > k_0} \frac{1-|\lambda_{n_k,2}|^2}{|\lambda_{n_k,2}-\zeta_{r}|^{2m_{r}}},
    \end{align*}
    which gives again a contradiction.
\end{proof}

It remains to prove the implication $(iii) \implies (i)$ of Theorem \ref{main_thm}, i.e.\ deducing the multiplier interpolation result from the Carleson condition and \eqref{main_sum}. As we will see later, given a bounded sequence, our construction of a function in $\cM(\cH(b))$ that interpolates such a sequence involves the Blaschke product $B_{\Lambda}$. Our next aim will thus be to prove that \eqref{main_sum} implies $B_{\Lambda}\in\cH(b)$. In order to keep the computations more tractable we first establish this result for the case when the Pythagorean mate of $b$ is given by $a(z)=a_1(z)=\frac{(z-1)^N}{2^N}$. The general case will be deduced in Lemma \ref{B_Lambda in H(b)}. According to the example mentioned
in the introduction, when $N=2$, then the
corresponding $b_1$ is given by
$b_1=c(1+z)(z-(3+s\sqrt{2})$ (see \cite{Fricain16}). So let us consider $b_1$, the outer function such that $a_1$ is its Pythagorean mate. In this scenario, \eqref{main_sum} in Theorem \ref{main_thm}(iii) is equivalent to 
\begin{equation}\label{main_sum_1}
    \sum_{n \in \bN} \frac{1-|\lambda_n|^2}{|1-\lambda_n|^{2N}} < \infty.
\end{equation}
We introduce the following functions:
\begin{align*}
    G_N(z) & := \frac{\varphi_n(z) - T_{N-1}(\varphi_n,1)(z)}{(z-1)^N}, \\
    H_N(z) & := \frac{\varphi_n(z)f(z) - T_{N-1}(\varphi_n f,1)(z)}{(z-1)^N},
\end{align*}
 where $\varphi_n(z)=\frac{z-\lambda_n}{1-\overline{\lambda_n}z}$ is the elementary Blaschke factor (M\"obius transform) at $\lambda_n$.

\begin{lem} \label{D_1^N(phi_n)}
    Let $\varphi_n$ be the Blaschke factor at $\lambda_n$. Then 
    \[
    D_1^N(\varphi_n) = \frac{1-|\lambda_n|^2}{|1-\lambda_n|^{2N}} |\lambda_n|^{2(N-1)}.
    \]
\end{lem}

\begin{proof}
    Clearly,
    \[
    \varphi_n^{(j)}(z) = j ! \ol{\lambda_n}^{j-1}\frac{1-|\lambda_n|^2}{(1-\ol{\lambda_n}z)^{j+1}},
\quad j\ge 1 .
    \]
    In particular we have
    \[
    \varphi_n^{(j)}(1) = j ! \ol{\lambda_n}^{j-1}\frac{1-|\lambda_n|^2}{(1-\ol{\lambda_n})^{j+1}}.
    \]
    We now proceed to computing $D_1^N (\varphi_n)$. We note that $D_1^N (\varphi_n)=\|G_N\|^2_{H^2}$ and we claim that 
    \[
    G_N(z)= \frac{1-|\lambda_n|^2}{(1-\ol{\lambda_n})^N} \ol{\lambda_n}^{N-1} k_{\lambda_n} (z),
    \]
    where $k_{\lambda_n} (z)$ is the Sz\"ego kernel at $\lambda_n$. We proceed by induction. From \cite{Serra03}*{Lemma 2.1} (or by direct computation) we know that the claim is true for $N=1$. Now suppose the claim is true for $N$. Then
    \begin{align*}
        G_{N+1} (z) & = \frac{\varphi_n(z) - \sum_{j=0}^{N-1} \frac{\varphi_n^{(j)}(1)}{j!}(z-1)^{j} -\frac{\varphi_n^{(N)}(1)}{N!}(z-1)^{N}  }{(z-1)^{N+1}} \\
 & = \frac{\varphi_n(z) - \sum_{j=0}^{N-1} \frac{\varphi_n^{(j)}(1)}{j!}(z-1)^{j} - \frac{\ol{\lambda_n}^{N-1}(1-|\lambda_n|^2)}{(1-\ol{\lambda_n})^{N+1}} (z-1)^N }{(z-1)^{N+1}} \\
        & = \frac{G_N (z)}{z-1} - \frac{\ol{\lambda_n}^{N-1}(1-|\lambda_n|^2)}{(1-\ol{\lambda_n})^{N+1}(z-1)} \\
        & = \frac{1}{z-1} \Big( G_N(z) - \frac{\ol{\lambda_n}^{N-1}(1-|\lambda_n|^2)}{(1-\ol{\lambda_n})^{N+1}} \Big) \\
        & = \frac{1}{z-1} \Big( \frac{1-|\lambda_n|^2}{(1-\ol{\lambda_n})^N} \ol{\lambda_n}^{N-1} k_{\lambda_n} (z) - \frac{\ol{\lambda_n}^{N-1}(1-|\lambda_n|^2)}{(1-\ol{\lambda_n})^{N+1}} \Big) \\
       &=  \frac{1}{z-1} \frac{\ol{\lambda_n}^{N-1}(1-|\lambda_n|^2)}{(1-\ol{\lambda_n})^{N+1}}
 \left( \frac{1-\ol{\lambda_n} }{1-\ol{\lambda_n}z} -1 \right)  \\ 
	& = \frac{1-|\lambda_n|^2}{(1-\ol{\lambda_n})^{N+1}} \ol{\lambda_n}^{N} k_{\lambda_n} (z).
    \end{align*}
    We can now easily compute $D_1^N(\varphi_n)$:
    \[
    D_1^N (\varphi_n) = \|G_N\|_{H^2}^2 = \frac{1-|\lambda_n|^2}{|1-\lambda_n|^{2N}} |\lambda_n|^{2(N-1)},
    \]
    which ends the proof.
\end{proof}

\begin{lem} \label{D_1^N (phi_n f)}
    Let $\varphi_n$ be the Blaschke factor at $\lambda_n$. Then there exists $C>0$ that depends only on $N$ such that for every $f \in \cH(b_1)$ we have
    \[
    D_1^N(\varphi_n f) \leq D_1^N(f) + C\max \Big( \big\{ |f^{(j)}(1)/j!|^2 : j=0,\ldots,N-1 \big\} \Big) \frac{1-|\lambda_n|^2}{|1-\lambda_n|^{2N}}.
    \]
\end{lem}

\begin{proof}
Define
\[
U_N:= \sum_{j=0}^{N-1} \frac{f^{(j)}(1)}{j!} G_{N-j}(z) + \varphi_n (z) \frac{f(z)-T_{N-1}(f,1)(z)}{(z-1)^N}
\]
Then we have
\begin{align*}
    (z-1)U_{N+1}(z)-U_N & = (z-1) \Big( \sum_{j=0}^N \frac{f^{(j)}(1)}{j!} G_{N-j+1}(z) + \varphi_n(z) \frac{f(z)-T_{N}(f,1)(z)}{(z-1)^{N+1}}\Big)\\
    & \quad  - \Big( \sum_{j=0}^{N-1} \frac{f^{(j)}(1)}{j!} G_{N-j+1}(z) + \varphi_n(z) \frac{f(z)-T_{N-1}(f,1)(z)}{(z-1)^{N}}\Big)\\
    & = \sum_{j=0}^{N-1} \frac{f^{(j)}(1)}{j!}\Big((z-1)G_{N-j+1}(z) -G_{N-j}(z) \Big)\\
    & \quad + (z-1) \frac{f^{(N)}(1)}{N!}G_1(z) + \varphi_n(z)\frac{T_{N-1}(f,1)(z)-T_{N}(f,1)(z)}{(z-1)^N} \\
    & = -\sum_{j=0}^{N-1} \frac{f^{(j)}(1)}{j!} \frac{\varphi_n^{(N-j)}(1)}{(N-j)!} + (z-1)\frac{f^{(N)}(1)}{N!} \frac{\varphi_n(z)-\varphi_n(1)}{z-1} \\
    & \quad - \frac{\varphi_n(z)}{(z-1)^N}(z-1)^N \frac{f^{(N)}(1)}{N!}\\
    & = \sum_{j=0}^N \frac{f^{(j)}\varphi_n^{(N-j)}(1)}{j!(N-j)!} = - \frac{(f\varphi_n)^{(N)}(1)}{N!}.
\end{align*}
Clearly,
\[
(z-1)H_{N+1}(z)-H_N(z) = - \frac{(f\varphi_n)^{(N)}(1)}{N!}.
\]
Since $H_1(z)=U_1(z)$, by induction we have that $H_N(z)=U_N(z)$ for all $N$. Note that $G_N$ is a scalar multiple of the reproducing kernel $k_{\lambda_n}$ and is thus orthogonal in $H^2$ to the shift invariant subspace $\varphi_nH^2$. Note also that $\varphi_n$, as a Blaschke factor, is inner and hence an isometric multiplier on $H^2$. Therefore from Lemma \ref{D_1^N(phi_n)} we obtain
    \begin{align*}
        D_1^N (\varphi_n f) & = \|H_N\|_{H^2}^2 \\
        & =  \Big| \sum_{j=0}^{N-1} \frac{f^{(j)}(1)}{j!} \frac{1-|\lambda_n|^2}{(1-\ol{\lambda_n})^{N-j}} \ol{\lambda_n}^{N-1-j}  \Big|^2 \frac{1}{1-|\lambda_n|^2} +D_1^N (f) \\
        & = \frac{1-|\lambda_n|^2}{|1-\lambda_n|^{2N}} \Big| \sum_{j=0}^{N-1} \frac{f^{(j)}(1)}{j!} (1-\ol{\lambda_n})^j \ol{\lambda_n}^{N-1-j} \Big|^2 +D_1^N (f)  \\
        & \leq C \max \Big( \big\{ |f^{(j)}(1)/j!|^2 : j=0, \ldots, N-1 \big\} \Big) \frac{1-|\lambda_n|^2}{|1-\lambda_n|^{2N}}+D_1^N (f) .
    \end{align*}
\end{proof}

By an easy induction, the following corollary is an immediate consequence of Lemma \ref{D_1^N (phi_n f)} and Theorem \ref{derivative B_Lambda} (which relates somehow the sum \eqref{main_sum_1} to the derivatives of the Blaschke product at $\zeta$).

\begin{cor}
    Suppose that \eqref{main_sum_1} holds and consider
    \[
    B_n (z) = \prod_{k=1}^n \frac{|\lambda_k|}{\lambda_k} \varphi_k(z).
    \]
    Then there exists $C>0$ which depends on $N$ and on the sum \eqref{main_sum_1} such that
    \[
    D_1^N(B_n) \leq C \sum_{k=1}^n \frac{1-|\lambda_k|^2}{|1-\lambda_k|^{2N}}.
    \]
\end{cor}

\begin{lem} \label{B_Lambda in H(b)}
    If \eqref{main_sum_1} holds, then $B_\Lambda \in \cH(b_1)$ and $|B_\Lambda(1)|=1$.
\end{lem}

Observe that the above conclusion trivially holds for every subproduct $B_{\Lambda'}$ with $\Lambda'\subset \Lambda$.

\begin{proof}
It is well known that $B_n \to B_\Lambda$ in $H^2$ as $n \to \infty$. This also implies that some subsequence of $(B_n)$ converges for almost every $\zeta$ to $B_{\Lambda}$, say $(B_{n_k})$. 
Define
\[
 \psi_{k,0}(\zeta)= \frac{B_{n_k}(\zeta)-T_{N-1}(B_{n_k},1)(\zeta)}{(\zeta-1)^N},\quad \zeta\in\bT\setminus\{1\},
\]
and $\psi_k=|\psi_{k,0}|$.

By the preceding observation, $B_{n_k}(\zeta)\to B_{\Lambda}(\zeta)$ a.e. Moreover, in view of Theorem \ref{derivative B_Lambda} and the remark thereafter, the Taylor polynomial $T_{N-1}(B_{n_k},1)(\zeta)$ converges at every $\zeta$ to $T_{N-1}(B_{\Lambda},1)(\zeta)$. Hence 
\[
 \psi_k(\zeta)\to \psi(\zeta)=\Big| \frac{B_{\Lambda}(\zeta)-T_{N-1}(B_{\Lambda},1)(\zeta)}{(\zeta-1)^N} \Big|,\quad \text{ a.e. } \zeta\in\bT \text{ when }k\to+\infty.
\]

Then we can use the above corollary and Fatou's lemma to obtain
\begin{align*}
    D_1^N(B_\Lambda) & 
     = \int_{\bT} \Big| \frac{B_{\Lambda}(\zeta)-T_{N-1}(B_{\Lambda},1)(\zeta)}{(\zeta-1)^N} \Big|^2 dm(\zeta)=\int_{\bT}|\psi(\zeta)|^2dm(\zeta)
    =\int_{\bT} \lim_{k\to\infty}|\psi_k(\zeta)|^2dm(\zeta)\\
 &\le \liminf_{k\to+\infty}\int_{\bT} |\psi_k(\zeta)|^2dm(\zeta)\\
    & \leq \liminf_{k \to \infty} \int_{\bT} \Big| \frac{B_{n_k}(\zeta)-T_{N-1}(B_{n_k},1)(\zeta)}{(\zeta-1)^N} \Big|^2 dm(\zeta) \\
    & \leq \sup_k D_1^N(B_{n_k}) \\
    & \leq C \sum_{j=1}^\infty \frac{1-|\lambda_j|}{|1-\lambda_j|^{2N}},
\end{align*}
which proves the first statement. 

The second statement is already known from Frostman's theorem (see for instance \cite[Theorem 1]{Ahern71}).
\end{proof}

\begin{rem*}
The expression of $\psi_{k,0}$ in the proof of Lemma \ref{B_Lambda in H(b)} is in the spirit of a formula appearing already in \cite{Fricain08}*{p.2123} and valid in the context of the upper half plane in order to show that the corresponding function is in $\cH(b) \subset H^2$. There, the authors represent $\psi_{k,0}$ as a linear combination of derivatives of kernels. The situation does not immediately transfer to the unit disk by a simple change of variables and it would be interesting to explore this further in the setting of the unit disc to obtain a similar representation in $\bD$.
\end{rem*}

To consider the general case of several points $\zeta_j \in \bT$, we will need a result on intersections of $M(\bar a)$-spaces which seems to be interesting on its own.

\begin{prop} \label{prop: intersection M(bar a)}
Let $a_1,\ldots a_l \in H^\infty$ satisfy the following corona type condition, i.e. there exists $\delta >0$ such that
for all $1\le i,j\le n$, $i\neq j$,
\[
 |a_i(z)|+|a_j(z)|  \geq \delta \quad \forall z \in \bD.
\]
Then 
\[
\bigcap_{j=1}^l M(\ol{a_j}) = M (\prod_{j=1}^l \ol{a_j}).
\]
\end{prop}

\begin{proof}

It is enough to show the claim for $l=2$. The general case follows by induction.

Clearly (see also \cite[Theorem 16.7]{Fricain15a}), 
\begin{align*}
M(\ol{a_1 a_2}) & = T_{\ol{a_1a_2}}H^2 = T_{\ol{a_1}}(T_{\ol{a_2}}H^2)\\
& = T_{\ol{a_2}}(T_{\ol{a_1}}H^2) \subset T_{\ol{a_1}}H^2 \cap T_{\ol{a_2}}H^2\\
& = M(\ol{a_1}) \cap M(\ol{a_2}).
\end{align*}
For the reverse inclusion, note that by the corona theorem there exist $h_1, h_2 \in H^\infty$ such that
\[
a_1h_1 + a_2 h_2 \equiv 1.
\]
Let $f \in M(\ol{a_1}) \cap M(\ol{a_2})$, then there exist $g_1,g_2 \in H^2$ such that $f=T_{\ol{a_j}}g_j$, $j=1,2$, and hence
\begin{align*}
    T_{\overline{a_1a_2}}(T_{\bar h_1} g_2 + T_{\bar h_2} g_1) 
&= T_{\overline{a_1a_2} \bar h_1} g_2 + T_{\overline{a_1a_2} \bar h_2} g_1 \\
&= T_{\bar a_1 \bar h_1} T_{\bar a_2} g_2 + T_{\bar a_2 \bar h_2} T_{\bar a_1} g_1 \\
&= T_{\bar a_1 \bar h_1 + \bar a_2 \bar h_2} f \\
&= f,
\end{align*}
and hence $f=T_{\overline{a_1a_2}}g$ with $g=T_{\bar h_1} g_2 + T_{\bar h_2} g_1\in H^2$,
so that $f\in M(\ol{a_1 a_2})$.
\end{proof}

\begin{cor} \label{B_Lambda in H(b) general}
    If $\Lambda$ satisfies \eqref{main_sum}, then $B_\Lambda \in \cM(\cH(b))$.
\end{cor}

\begin{proof}
    By Lemma \ref{B_Lambda in H(b)} and a rotation argument, for all $j$ we have $D_{\zeta_j}^N(B_\Lambda) < \infty$. Hence $B_\Lambda \in \bigcap_j \mathcal{H}(b_j)=\bigcap_j M(\ol a_j)$, where $b_j$ are the functions such that $a_j(z)=(z-\zeta_j)^{m_j}/2^{m_j}$ are the corresponding Pythagorean mates. Since the functions $a_1,\ldots,a_l$ clearly satisfy the corona condition, we obtain the result from Proposition \ref{prop: intersection M(bar a)}.
\end{proof}

Before showing the multiplier interpolation result, let us make the following useful observation. It is always possible to add a finite number of points to a (multiplier-) interpolating sequence (once the equivalence of interpolation and multiplier interpolation is established, it is clear that the result is true also in $\cH(b)$). It is of course enough to show that we can add one point.

\begin{lem}\label{AddPts}
Let $b$ be a rational, non inner function with $\|b\|_{H^{\infty}}=1$. Let $\Lambda=\{\lambda_n\}_{n\ge 1}$ be $\cM(\cH(b))$-interpolating, and $\lambda_0\in\bD\setminus\Lambda$. Then $\Lambda\cup \{\lambda_0\}=\{\lambda_n\}_{n\ge 0}$ is $\cM(\cH(b))$-interpolating.
\end{lem}

\begin{proof}
Let $(v_n)_{n\ge 0}\in \ell^{\infty}$. Since $\Lambda$ is interpolating there is $F\in \cM(\cH(b))$ such that $F(\lambda_n)=v_n$, $n\ge 1$. Setting now
\begin{equation}\label{FormulaAddPts}
 f=F+\frac{a B_{\Lambda}}{a(\lambda_0) B_{\Lambda}(\lambda_0)}(v_0-f_0(\lambda_0))
\end{equation}
we obtain obviously a function $f\in \cM(\cH(b))$ with $f(\lambda_n)=v_n$, $n\ge 0$. 
\end{proof}


\begin{lem} \label{sum&C_implies_mult_interp}
    Let $b$ be a rational, non inner function with $\|b\|_{H^{\infty}}=1$. If $\Lambda$ satisfies \eqref{main_sum} and \eqref{Carleson_condition}, then $\Lambda$ is interpolating for $\cM (\cH (b))$. Moreover, the interpolating function can be chosen in $aH^2 \cap H^\infty$.
\end{lem}

\begin{proof}
Let $(v_n)_{n \in \bN} \in \ell^\infty$. We have to construct a function $F \in \cM(\cH(b))$ such that $f(v_n)=a_n$ for every $n \in \bN$.

    Observe first that from Corollary \ref{B_Lambda in H(b) general}, we know that $B_\Lambda \in \cH(b)$. Hence
    \begin{equation}\label{FormB}
    B_\Lambda(z) = \prod_{j=1}^l(z-\zeta_j)^{m_j} g(z) + p_N(z)=a(z)g(z) + p_N(z),
    \end{equation}
    where $g \in H^2$. Note in particular that $|p_N(\zeta_j)| = |B_\Lambda (\zeta_j)| = 1$, $j=1,\ldots,l$.

Let $\cZ(p_N)$ be the zero set of $p_N$ and set $\{\eta_r\}_{r=1}^{R}=\cZ(p_N)\cap\bT$. Observe that $0\le R\le N-1$ since $p_N$ is not the zero polynomial. In view of $p_N(\zeta_j)\neq 0$, we also have $|\eta_r-\zeta_j|\ge \delta$ for some $\delta>0$ and every $r$ and $j$. Hence the set
\[
 \Lambda_1:=\{\lambda_{n,1}\}:=\bigcup_{r=1}^R D(\eta_r,\delta/2)\cap \Lambda
\]
is far from $\zeta_j$, $j=1,\ldots, l$.  Since $\Lambda_1$ satisfies \eqref{Carleson_condition} and $|\lambda_{n,1}-\zeta_j|\ge \delta/2$ for every $n$ and $j$, there exists $h_1 \in H^\infty$ such that 
    \[
    h_1(\lambda_{n,1})=\frac{v_{n,1}}{a(\lambda_{n,1})}.
    \]
    Then define $h(z)=a(z) h_1(z) \in \cH(b) \cap H^\infty$ which solves the interpolation problem $h(\lambda_{n,1})=v_{n,1}$.

The remaining sequence $\Lambda\setminus\Lambda_1$ might contain zeros of $p_N$, so let us first consider $\Lambda_2=\Lambda\setminus (\Lambda_1\cup \cZ(p_N))$ the elements of which we will denote by $\lambda_{n,2}$. Then for every $n\in\bN$, 
we have $|p_N(\lambda_{n,2})|\ge \varepsilon$ for some $\varepsilon>0$ ($\Lambda_2$ does not contain nor accumulate to any zero of $p_N$). Note also that $|B_{\Lambda_1}(\lambda_{n,2})|\ge |B_{\Lambda\setminus\lambda_{n,2}}(\lambda_{n,2})|\ge \eta>0$ by the Carleson condition. Hence the sequence
    \[
    u_n:=-\frac{v_{n,2} - h(\lambda_{n,2})}{B_{\Lambda_1}(\lambda_{n,2}) p_N(\lambda_{n,2})}
    \]
    is uniformly bounded, so that there exists $f_1 \in H^\infty$ with $f_1(\lambda_{n,2}) = u_n$. Then, in view of \eqref{FormB}, we have
    \[
    f(z)=(B_\Lambda(z) - p_N(z))f_1(z) \in aH^2\cap H^{\infty}\subset \cH (b) \cap H^\infty,
    \]
    and by construction
    \[
    f(\lambda_{n,2}) = \frac{v_{n,2} - h(\lambda_{n,2})}{B_{\Lambda_1}(\lambda_{n,2})}.
    \]
    We can now define $F=B_{\Lambda_1}f + h$. Since $f,h\in aH^2\cap H^{\infty}$ and $B_{\Lambda_1}\in H^{\infty}$, we get $F\in aH^2\cap H^{\infty}\subset \cH(b) \cap H^\infty$, and by construction $F(\lambda_n)=v_n$ for $\lambda_n\in \Lambda_3:=\Lambda_1\cup\Lambda_2$.

As a result, $\Lambda_3$ is $\cM(\cH(b))$-interpolating, and the interpolating function $F$ is in $aH^2\cap H^{\infty}$.

Finally, from Lemma \ref{AddPts}, since $\Lambda_0:=\cZ(p_N)\cap\Lambda$ is finite, it follows that $\Lambda=\Lambda_3\cup \Lambda_0$ is $\cM(\cH(b))$-interpolating. And moreover, since $F\in aH^2\cap H^{\infty}$, formula \eqref{FormulaAddPts} shows that the corresponding interpolating function on $\Lambda$ is also in $aH^2\cap H^{\infty}$.
\end{proof}

\section{Random interpolation: proof of Theorem \ref{main_random}}

\begin{lem} \label{random lemma}
    Let $\Lambda$ be a random sequence, $\zeta \in \bT$ and $M \in \bN$. Then
    \[
    \bP\Big( \sum_{n \in \bN} \frac{1-|\lambda_n|^2}{|\zeta-\lambda_n|^{2M}} < \infty \Big) = 
    \begin{dcases}
        1 & \text{if }\sum_{n \in \bN} (1-|\lambda_n|^2)^{\frac{1}{2M}} < \infty \\
        0 & \text{otherwise}
    \end{dcases}
    \]
\end{lem}

\begin{proof}
    Since the random sequences we are considering are invariant by rotations, we can assume without loss of generality that $\zeta=1$.
    
    Let us introduce
    \[
    X_n= \frac{1-|\lambda_n|^2}{|1-\lambda_n|^{2M}}.
    \]
    By construction, $X_n$ are independent. We start noting that
    \begin{align*}
    \bP(X_n > 1) & = \bP \Big( (1-|\lambda_n|^2)^{1/M} > 1+ |\lambda_n|^2 -2|\lambda_n|\cos(\vartheta_n) \Big) \\
    & = \bP \Big( \cos(\vartheta_n) > \frac{1+|\lambda_n|^2 - (1-|\lambda_n|^2)^{1/M}}{2|\lambda_n|} \Big).
    \end{align*}
    Define
    \[
    u_n := \frac{1+|\lambda_n|^2 - (1-|\lambda_n|^2)^{1/M}}{2|\lambda_n|},
    \]
    then we obtain
    \begin{align*}
        \bP(X_n > 1) & = \bP(0 < \vartheta_n < \arccos (u_n)) + \bP(2\pi - \arccos(u_n) < \vartheta_n < 2\pi ) \\
        & = \frac{\arccos(u_n)}{\pi}.
    \end{align*}
    For $n \to \infty$ we have that $u_n \to 1$ which implies $\arccos (u_n) \sim (1-|u_n|^2)^{1/2}$. A computation yields 
    \[
    1-|u_n|^2 \sim (1-|\lambda_n|^2)^{1/M},
    \]
    which finally implies
    \[
    \bP(X_n > 1) \sim \frac{(1-|\lambda_n|^2)^{\frac{1}{2M}}}{\pi}.
    \]
    We obtain that if $\sum_n (1-|\lambda_n|^2)^{\frac{1}{2M}}=\infty$ then $\sum_{n\in\bN}\bP(X_n>1)=+\infty$. By Theorem \ref{kolmogorov 3-series} (or directly by the Borel-Cantelli lemma), we conclude that $\bP(\sum_n X_n < \infty)=0$, which yields the second part of the claim.
\\

Let us now consider the case when $ \sum_n (1-|\lambda_n|^2)^{\frac{1}{2M}}<\infty$. The sum condition translates immediately to $\sum_{n\in\bN}\bP(X_n>1)<+\infty$, which already yields condition (i) of Theorem \ref{kolmogorov 3-series}.
In order to check the conditions (ii) and (iii), let us introduce
    \[
   Y_n = \frac{1-|\lambda_n|^2}{|1-\lambda_n|^{2M}} \chi_{\{X_n \leq 1\}}.
\]

    We proceed computing $\bE[Y_n]$. Since $\vartheta\mapsto |1-|\lambda_n|e^{i\vartheta}|$ is even, we get
    \begin{align}\label{Cond(ii)}
        \bE[Y_n] & = \frac{1}{2\pi}\int_{\arccos(u_n)}^{2\pi-\arccos(u_n)} \frac{1-|\lambda_n|^2}{|1-|\lambda_n|e^{i\vartheta}|^{2M}} d\vartheta 
         = \frac{1}{\pi} \int_{\arccos(u_n)}^{\pi} \frac{1-|\lambda_n|^2}{|1-|\lambda_n|e^{i\vartheta}|^{2M}} d\vartheta \nonumber\\
& = \frac{1}{\pi} \int_{\arccos(u_n)}^{\pi/2} \frac{1-|\lambda_n|^2}{|1-|\lambda_n|e^{i\vartheta}|^{2M}} d\vartheta 
 + \frac{1}{\pi}\int_{\pi/2}^{\pi} \frac{1-|\lambda_n|^2}{|1-|\lambda_n|e^{i\vartheta}|^{2M}} d\vartheta.
    \end{align}

Let us discuss the first integral. 
Observe that $2|1-re^{i\vartheta}|^2\ge |1-e^{i\vartheta}|^2=2(1-\cos(\vartheta))$ for $r\in (0,1)$ and $\vartheta\in [0,\pi/2]$. Hence
    \begin{align*}
         \frac{1}{\pi} \int_{\arccos(u_n)}^{\pi/2} \frac{1-|\lambda_n|^2}{|1-|\lambda_n|e^{i\vartheta}|^{2M}} d\vartheta & \le \frac{1}{\pi} \int_{\arccos(u_n)}^{\pi/2} \frac{1-|\lambda_n|^2}{(1-\cos(\vartheta))^{M}} d\vartheta \\
         & = \frac{2(1-|\lambda_n|^2)}{2^M\pi} \int_{c_n}^{1/2} \frac{1}{t^{2M}} \frac{dt}{\sqrt{1-t^2}},
    \end{align*}
    where in the last equality we have used the change of variable $t=\sin(\vartheta/2) $ and $c_n=\sin(\frac{\arccos(u_n)}{2})$. If $t \in [c_n,1/2]$, then the function $t \mapsto \sqrt{1-t^2}$ is bounded from above and below. This yields
    \begin{align*}
        \frac{2(1-|\lambda_n|^2)}{2^M\pi} \int_{c_n}^{1/2} \frac{1}{t^{2M}} \frac{dt}{\sqrt{1-t^2}} & \simeq \frac{2(1-|\lambda_n|^2)}{2^M \pi}
        \int_{c_n}^{1/2} \frac{1}{t^{2M}} dt \\
        & = \frac{2(1-|\lambda_n|^2)}{2^M \pi(1-2M)} (2^{2M-1} - \frac{1}{c_n^{2M-1}}).
    \end{align*}
    Since, as already shown, $\arccos(u_n) \sim (1-|\lambda_n|^2)^{1/2}$ for $n \to \infty$, we have that $c_n \sim (1-|\lambda_n|^2)^{\frac{1}{2M}}/2$. Hence
 \begin{align*}
        \frac{2(1-|\lambda_n|^2)}{2^M\pi} \int_{c_n}^{1/2} \frac{1}{t^{2M}} \frac{dt}{\sqrt{1-t^2}} 
&\simeq \frac{2(1-|\lambda_n|^2)}{2^M\pi(2M-1)} \left(\frac{1}{(1-|\lambda_n|^2)^{(2M-1)/(2M)}}- 2^{2M-1}  \right)\\
 &\simeq (1-|\lambda_n|^2)^{1/(2M)}
   \end{align*}
Then it is immediate that the first integral is asymptotic to $C(1-|\lambda_n|^2)^{\frac{1}{2M}}$ as $n \to \infty$.
Since the denominator of the second integral in \eqref{Cond(ii)} is comparable to a constant, this second integral behaves like $(1-|\lambda_n|^2)$, which is summable by the Blaschke condition. We thus obtain condition (ii) of the theorem:
\[
  \sum \bE[Y_n]\simeq  \sum_{n \in \bN} \Big((1-|\lambda_n|^2)^{\frac{1}{2M}}+(1-|\lambda_n|^2)\Big) \simeq \sum_{n \in \bN} (1-|\lambda_n|^2)^{\frac{1}{2M}}<+\infty.
\]

    In order to check condition (iii), it remains to estimate $\bV[Y_n]$. Note that by our previous discussions, whenever $\vartheta \in [\arccos(u_n),2\pi-\arccos(u_n)]$ then 
    \[
    \frac{1-|\lambda_n|^2}{|1-\lambda_n|^{2M}} = \frac{1-|\lambda_n|^2}{|1-|\lambda_n|e^{i\vartheta}|^{2M}} \leq 1.
    \]
    With this in mind, we obtain
    \begin{align*}
        \bE[Y_n^2] & = \frac{1}{2\pi}\int_{\arccos(u_n)}^{2\pi-\arccos(u_n)} \Bigg( \frac{1-|\lambda_n|^2}{|1-|\lambda_n|e^{i\vartheta}|^{2M}} \Bigg)^2 d\vartheta \\
        & \leq \frac{1}{2\pi}\int_{\arccos(u_n)}^{2\pi-\arccos(u_n)} \frac{1-|\lambda_n|^2}{|1-|\lambda_n|e^{i\vartheta}|^{2M}} d\vartheta \\
        & = \bE[Y_n],
    \end{align*}
    which yields
    \[
    \bV[Y_n] \leq \bE[Y_n](1-\bE[Y_n]) \simeq \bE[Y_n].
    \]
Consequently
\[
 \sum \bV[Y_n]\le \sum \bE[Y_n]<+\infty,
\]
which is condition (iii).
    By Theorem \ref{kolmogorov 3-series} we conclude that
    \[
    \sum_{n\in\bN} X_n=\sum_{n \in \bN} \frac{1-|\lambda_n|^2}{|\zeta-\lambda_n|^{2M}} < \infty \quad \text{almost surely.}
    \]
\end{proof}

\begin{proof}[Proof of Theorem \ref{main_random}]
     Suppose first that $\sum_n (1-|\lambda_n|^2)^{\frac{1}{2M}} < \infty$.
From \cite{cochran90}*{Theorem 2} and \cite{rudowicz94}*{Theorem p.160} we know that a random sequence $\Lambda$ satisfies the Carleson condition almost surely if and only if
    \[
    \sum_{k \in \bN} N_k^2 2^{-k} < \infty.
    \]
       Since $M \geq 1$, then $((1-|\lambda_n|^2)^{\frac{1}{2M}})_{n \in \bN} \in \ell^1$ implies $(N_k 2^{-k/2})_{k \in \bN} \in \ell^1$. Consequently $(N_k 2^{-k/2})_{k \in \bN} \in \ell^2$, so the random sequence $\Lambda$ almost surely satisfies the Carleson condition. 

Moreover, Lemma \ref{random lemma} implies that for arbitrary
$\zeta \in \bT$
    \[
    \sum_{n \in \bN} \frac{1-|\lambda_n|^2}{|\zeta-\lambda_n|^{2M}} < \infty \quad \text{almost surely.}
    \]
By hypothesis, $m_j\le M$, $1\le j \le l$ so that 
   \[
    \sum_{n \in \bN} \frac{1-|\lambda_n|^2}{|\zeta_j-\lambda_n|^{2m_j}} < \infty \quad \text{almost surely for }1\le j\le l.
    \]
We conclude with Theorem \ref{main_thm}, that $\Lambda$ is almost surely interpolating (multiplier or universal).
\\

If $\sum (1-|\lambda_n|^2)^{\frac{1}{2M}} = \infty$, then, applying Lemma \ref{random lemma} once more, there exists $r$, $1\leq r \leq l$ such that
    \[
    \sum_{n \in \bN} \frac{1-|\lambda_n|^2}{|\zeta_{r}-\lambda_n|^{2m_{r}}} = \infty \quad \text{almost surely.}
    \]
    This means that \eqref{main_sum} does almost surely not hold and so we get
    \[
    \bP(\Lambda \text{ is interpolating for } \cH(b))=0.
    \]
\end{proof}

\bibliographystyle{habbrv}
\bibliography{biblio}

\end{document}